\documentclass[a4paper,11pt,english,leqno]{amsart}
\usepackage{amsmath, amsthm, amsfonts,amssymb, mathrsfs}
\usepackage{verbatim, graphicx, ifthen, enumitem}
\usepackage[T1]{fontenc}
\usepackage[dvipsnames]{xcolor}
\usepackage{hyperref}
\usepackage{cleveref}
\usepackage{tikz}
\usepackage{bm}
\usepackage{caption}
\usepackage{subcaption}
\usepackage[english]{babel}
\usetikzlibrary{calc}
\usepackage[normalem]{ulem}

\newcommand\myshade{85}
\colorlet{mylinkcolor}{violet}
\colorlet{mycitecolor}{YellowOrange}
\colorlet{myurlcolor}{Aquamarine}

\hypersetup{
  linkcolor  = mylinkcolor!\myshade!black,
  citecolor  = mycitecolor!\myshade!black,
  urlcolor   = myurlcolor!\myshade!black,
  colorlinks = true,
}

\newtheorem{thm}{Theorem}[subsection]
\newtheorem{cor}[thm]{Corollary}

\theoremstyle{definition}
\newtheorem{defn}[thm]{Definition}

\theoremstyle{remark}
\newtheorem{rem}[thm]{Remark}

\numberwithin{equation}{subsection}
\numberwithin{figure}{subsection}

\newcommand{\diff}{\mathrm{d}}
\newcommand{\C}{{\mathbb C}}
\newcommand{\R}{{\mathbb R}}

\newcommand{\Z}{{\mathbb Z}}
\newcommand{\Eop}{{\mathbf{E}}}

\newcommand{\imag}{\mathrm{i}}
\newcommand{\e}{\mathrm{e}}

\newcommand{\Lop}{{\mathbf L}}
\newcommand{\Tope}{{\mathbf T}}

\newcommand{\Ordo}{\mathrm{O}}

\newcommand\heatkernel{\beta}
\newcommand\Hspace{\mathscr{H}}

\newcommand{\Dop}{\mathbf{D}}
\newcommand{\dop}{\boldsymbol\partial}

\DeclareMathOperator{\re}{Re}
\DeclareMathOperator{\im}{Im}

\makeatletter
\let\@@citation@@=\citation

\renewcommand{\citation}[1]{\@@citation@@{#1}%
\@for\@tempa:=#1\do{\@ifundefined{cit@\@tempa}%
  {\global\@namedef{cit@\@tempa}{}}{}}%
}

\def\@lbibitem[#1]#2#3\par{%
  \@ifundefined{cit@#2}{}{\item[\@biblabel{#1}\hfill]}%
  \if@filesw
      {\let\protect\noexpand
       \immediate
       \write\@auxout{\string\bibcite{#2}{#1}}}\fi\ignorespaces
  \@ifundefined{cit@#2}{}{#3}}
\def\@bibitem#1#2\par{%
  \@ifundefined{cit@#1}{}{\item}%
  \if@filesw \immediate\write\@auxout
    {\string\bibcite{#1}{\the\value{\@listctr}}}\fi\ignorespaces
  \@ifundefined{cit@#1}{}{#2}}
\makeatother

\begin{document}

%
\title{Spectral interpretation of Riemann zeta zeros}


\author{Haakan Hedenmalm}
\address{
Hedenmalm: 
Department of Mathematics and Computer Science
\\
St Petersburg State University
\\
St Petersburg, Russia
\\
Beijing Institute for Mathematical Sciences and Applications
\\
Huairou District, Beijing 101408, China
\\
Department of Mathematics and Statistics\\
University of Reading\\
RG6 6AX Reading, England}

\email{haakan00@gmail.com}


\subjclass[2020]{{}}

\begin{abstract} 
It is a well-known problem to identify the nontrivial zeros of the Riemann zeta function
in terms of an eigenvalue problem. We here find such an eigenvalue problem for
second order differential operators on the line. In a sense, our analysis pushes the analysis
of the zeta function over to the study of the Jacobi theta function, which may be thought of
as the fundamental solution of the heat (or Schr\"odinger) equation on the unit circle (or the
semi-infinite cylinder, if time is added). The eigenvalue problem takes the form
$\Lop\Dop u+\alpha \Lop u=0$, where $\Lop$ and $\Dop$ are first-order 
differential operators, of which only $\Lop$ involves the theta function. In a formal sense, 
then, $\alpha$ is an eigenvalue of the twisted operator $-\Lop\Dop\Lop^{-1}$. Based on this 
formal thinking, we develop the notion of self-adjointness of the pair $(\Lop\Dop,\Lop)$, to 
adapt the Hilbert-P\'olya idea to the spectral problem at hand.
\end{abstract}

\maketitle

\section{Jacobi theta function and related functions}\label{sec_INT}

\subsection{The Jacobi theta function}
\label{subsec-1.1}

The Jacobi theta function $\vartheta_{00}(z,\tau)$ is given by the expression
\[
\vartheta_{00}(z,\tau):=\sum_{n\in\Z}\e^{\imag 2\pi n z+\imag \pi n^2\tau},
\]
where $z,\tau\in\C$ with $\im\,\tau>0$. It solves the heat (or Schr\"odinger) equation
\begin{equation}
\imag\,\partial_\tau\vartheta_{00}(z,\tau)=\frac{1}{4}\,\partial_z^2\vartheta_{00}(z,\tau)
\label{eq:heat-1}
\end{equation}
and has the periodicities in the $z$ variable given by
\[
\vartheta_{00}(z+1,\tau)=\vartheta_{00}(z,\tau),\quad
\vartheta_{00}(z+\tau,\tau)=\exp(-\imag\pi(\tau+2z))\,\vartheta_{00}(z,\tau).
\]
There is also the remarkable identity found by Jacobi which says that
\begin{equation}
\vartheta_{00}(z,\tau)=\heatkernel(z,\tau)\,
\vartheta_{00}\bigg(\frac{z}{\tau},-\frac{1}{\tau}\bigg),
\label{eq:heat-2}
\end{equation}
where
\begin{equation}
\heatkernel(z,\tau):=(\tau/\imag)^{-\frac12}\exp\bigg(-\imag\frac{\pi z^2}{\tau}\bigg)
\label{eq:heat-3}
\end{equation}
is, up to a constant multiple, the heat kernel associated with 
\eqref{eq:heat-1}. This relation has an interpretation in terms of the classical Appell
transform, which will be mentioned as a remark toward the end. For the moment being, 
we shall focus on the $z=0$ instance of \eqref{eq:heat-1}, and write
$\vartheta_{00}(\tau):=\vartheta_{00}(0,\tau)$ to simplify the notation. Then 
\eqref{eq:heat-3} asserts that
\begin{equation}
\vartheta_{00}(\tau)=(\tau/\imag)^{-\frac12}\,
\vartheta_{00}\big(-1/\tau\big),\qquad \tau\in\C_{\im>0},
\label{eq:heat-4}
\end{equation}
which relation may be viewed as an instance of the Poisson summation formula.

\subsection{A function related to the Jacobi theta function}
We form the related expression 
\begin{equation}
\varTheta_{00}(\tau):=(\tau/\imag)^{\frac14}\bigg(\tau^2\vartheta_{00}''(\tau)+
\frac{3}{2}\tau\vartheta_{00}'(\tau)\bigg),
\label{eq:heat-5}
\end{equation}
which is holomorphic in $\C_{\im>0}$, and, as a consequence of \eqref{eq:heat-4}, it
is even-symmetric with respect to the inversion $\tau\mapsto-1/\tau$:
\begin{equation*}
\varTheta_{00}(-1/\tau)=\varTheta_{00}(\tau).
\end{equation*}
Although this fact is known, it is nice to see how it follows from the Jacobi identity
\eqref{eq:heat-4}.
By the functional relation \eqref{eq:heat-4}, the function 
\begin{equation*}
\vartheta_{00}^{\mathrm{sym}}(\tau):=(\tau/\imag)^{\frac14}\vartheta_{00}(\tau)
\end{equation*}
is even inverse-symmetric, that is,
\[
\vartheta_{00}^{\mathrm{sym}}(\tau)=\vartheta_{00}^{\mathrm{sym}}(-1/\tau),
\]
holds. Moreover, if $\dop^\times$ denotes the differential operator 
\begin{equation*}
\dop^\times f(\tau)=\tau f'(\tau),
\end{equation*}
it follows that if $F(\tau)$ is holomorphic in $\C_{\im>0}$, with $F(\tau)=F(-1/\tau)$,
then 
\begin{equation*}
\dop^\times F(\tau)=\tau F'(\tau)=\tau\,\tau^{-2} F'(-1/\tau)=-(\dop^\times F)(-1/\tau),
\end{equation*}
making $\dop^\times F$ odd inverse-symmetric. However, if we apply the differential 
operator twice, even inverse-symmetry gets preserved:
\begin{equation*}
(\dop^\times)^2 F(\tau)=[(\dop^\times)^2F](-1/\tau).
\end{equation*}
A direct calculation reveals that
\begin{equation*}
\varTheta_{00}(\tau)=(\dop^\times)^2\vartheta_{00}^{\mathrm{sym}}(\tau)-\frac{1}{16}
\vartheta_{00}^{\mathrm{sym}}(\tau),
\end{equation*}
so the even inverse-symmetry of $\varTheta_{00}(\tau)$ is immediate from that of 
$\vartheta_{00}^{\mathrm{sym}}(\tau)$, which in turn followed from the Jacobi identity. 
In view of the defining relation \eqref{eq:heat-5}, we have that 
\begin{multline}
\varTheta_{00}(\tau)=(\tau/\imag)^{\frac14}\bigg(\tau^2\vartheta_{00}''(\tau)+
\frac{3}{2}\tau\vartheta_{00}'(\tau)\bigg)
\\
=(\tau/\imag)^{\frac14}\bigg(-\pi^2\tau^2\sum_{n\in\Z}n^4\e^{\imag\pi n^2\tau}+
\frac{3}{2}\imag\pi\tau\sum_{n\in\Z}n^2\e^{\imag\pi n^2\tau}\bigg),
\label{eq:heat-6}
\end{multline}
and, consequently, for $0<t<+\infty$,
\begin{multline}
\varTheta_{00}(\imag t^2)
=t^{\frac12}\bigg(\pi^2 t^4\sum_{n\in\Z}n^4\e^{-\pi n^2 t^2}
-\frac{3}{2}\pi t^2\sum_{n\in\Z}n^2\e^{-\pi n^2 t^2}\bigg)
\\
=\pi t^{\frac92}
\sum_{n=1}^{+\infty}n^2(2\pi n^2-3t^{-2})\e^{-\pi n^2 t^2},
\label{eq:heat-7}
\end{multline}
a real-valued expression.
In the interval $1\le t<+\infty$, we have the estimates
\begin{equation}
0<\pi t^{\frac92}
\sum_{n=1}^{+\infty}n^2(2\pi n^2-3)\e^{-\pi n^2 t^2}
\le
\varTheta_{00}(\imag t^2)
\le\pi t^{\frac92}
\sum_{n=1}^{+\infty}2\pi n^4\e^{-\pi n^2 t^2},
\label{eq:heat-8}
\end{equation}
and by the even inversion-symmetry 
$\varTheta_{00}(\imag t)=\varTheta_{00}(\imag/t)$, we have on the remaining 
interval $0<t\le1$ the estimates
\begin{multline}
0<\pi t^{\frac94}
\sum_{n=1}^{+\infty}n^2(2\pi n^2-3)\e^{-\pi n^2/t^2}
\le\varTheta_{00}(\imag t^2)
\\
\le\pi t^{\frac92}
\sum_{n=1}^{+\infty}2\pi n^4\e^{-\pi n^2/t^2}.
\label{eq:heat-9}
\end{multline}
In particular, the function $\varTheta_{00}(\imag t^2)$ has Gaussian decay as
$t\to+\infty$ and by symmetry it also drops precipitously  as $t\to0^+$. 

\section{Observations by Riemann, P\'olya, Connes, and others}

\subsection{Dilations sum operator}
As observed, e.g., by Alain Connes \cite{Connes}, if $\Eop$ is the operator
defined by 
\begin{equation}
\Eop f(t):=\sum_{n=1}^{+\infty}f(nt),
\end{equation}
we may express $\varTheta_{00}(\imag t^2)$ in the form
\begin{equation}
\varTheta_{00}(\imag t^2)=t^{\frac12}\Eop h_{00}(t),\qquad 0<t<+\infty,
\end{equation}
where
\begin{equation}
h_{00}(t):=\frac{\pi}{2}(2\pi t^4-3t^2)\,\e^{-\pi t^2}.
\end{equation}
Curiously, the function $h_{00}$ is a linear combination of two Hermite functions
of degrees $0$ and $4$ with the defining property (up to constant multiples) that  
the integral over the real line vanishes (see Connes' remark \cite{Connes}). 
This means that the sum $\Eop h_{00}$,
while converging to a positive smooth function on $\R_{\ne0}$, must develop
a compensatory negative point mass at the origin in the sense of distribution theory. 
We will not explore this point further here.
However, we note that there is an analog of the Euler product formula for the 
Riemann zeta function which applies in the context of the operator $\Eop$:
\begin{equation}
\Eop=\prod_p \Eop^{\langle p\rangle},
\end{equation}
where $p$ runs over the primes, and 
\begin{equation}
\Eop^{\langle p\rangle}f(t)=\sum_{k=0}^{+\infty}f(p^kt).
\end{equation}
The operators in the product all commute, so there is no difficulty in interpreting the 
product, so long as the sums are meaningful on relevant function spaces.

\subsection{Riemann zeta zeros}
Riemann, in his classic 1859 work \emph{\"Uber die Anzahl der Primzahlen unter 
einer gegebener Gr\"osse} (see \cite{Riemann-works}),  studied the function 
$\varTheta_{00}(\imag t^2)$ because of the Mellin transform relationship
\begin{equation} 
\Xi(x)=\int_{0^+}^{+\infty}\varTheta_{00}(\imag t^2) \,t^{\imag x}\frac{\diff t}{t},\qquad
x\in\C,
\label{eq:Xi-1}
\end{equation}
where $\Xi(x)=\xi(\frac12+\imag x)$ and 
\begin{equation}
\xi(s)=\frac12\,s(s-1)\pi^{-s/2}\Gamma(s/2)\,\zeta(s),\qquad s\in\C.
\end{equation}
This formula appears also in Alain Connes' letter to Riemann \cite{Connes}, and constitutes 
the basis for George P\'olya's study of approximate $\Xi$-functions \cite{Polya}. 
It is clear from the formula that the zeros of $\Xi$ correspond to the zeros of the zeta function 
along the critical line $\re s=\frac12$, and they are all located along the real line should we trust 
the Riemann hypothesis. 

\section{Characterizing nontrivial zeta zeros as eigenvalues}

\subsection{Earlier work} Substantial effort has been made to understand the Riemann 
zeta zeros spectrally. For a comprehensive survey, we recommend the recent contribution
by Connes \cite{Connes}. We might nevertheless mention here the innovative approach by
Michael Berry and Jonathan Keating \cite{BerryKeating} which tries to adapt quantized
Hamiltonian dynamics to the required setting.

\subsection{Differential operators and differential equations}
We consider the differential operator
\begin{equation}
\Dop^\times f(t):=\frac{t}{\imag}\,f'(t),
\end{equation}
for differentiable functions $f$ on the real line. It is a quantization of the Hamiltonian 
$xp$, as envisaged by Berry and Keating \cite{BerryKeating}. Moreover, for a given 
$C^1$-smooth function $\varphi$, we consider the differential operators $\Lop_\varphi$ given 
by
\begin{equation}
\Lop_\varphi f(t)=\Dop^\times f(t)+f(t)\,\Dop^\times\varphi(t).
\end{equation}

\subsection{Characterization of critical line zeta zeros}

We now present a theorem which identifies the real roots of $\Xi$ as the eigenvalues (in an
extended sense) to a second order differential equation with vanishing boundary data.
We think that referring to them as eigenvalues is justified in the same way as in the 
Steklov eigenvalue problem (see, e.g., the survey paper by Girouard and Polterovich 
\cite{GirPolt}). 

We shall need the function
$\phi_{00}(t):=-\log\varTheta_{00}(\imag t^2)$ for $t\in\R_{>0}$, so that $\phi_{00}(t)$ is
real-valued with the inversion symmetry $\phi_{00}(1/t)=\phi_{00}(t)$, and, in addition, has 
the asymptotics
\[
\phi_{00}(t)=\pi t^2-\frac{9}{2}\log t-\log(2\pi^2)+\Ordo(t^{-2})\quad\text{as}\quad t\to+\infty.
\]

\begin{thm}
\label{thm:main-1}
For given $\alpha\in\R$, the solution  to the boundary value problem 
\begin{equation*}
\begin{cases}
\Lop_{\phi_{00}}\Dop^\times u+\alpha \Lop_{\phi_{00}}u=0\quad\text{on}\,\,\,\R_{>0},
\\
\lim_{t\to0^+}u(t)=\lim_{t\to+\infty}u(t)=0,
\end{cases}
\end{equation*}
is unique, i.e., $u=0$, unless $\alpha$ is such that $\Xi(\alpha)=0$. Moreover, if $\alpha\in\R$
has $\Xi(\alpha)=0$, then the solution is unique up to multiplicative constants, and of the form
$u=Cu_\alpha$, where $C$ is a constant and 
\[
u_\alpha(t)=t^{-\imag \alpha}\int_{0+}^{t}y^{\imag \alpha}\,
\varTheta_{00}(\imag y^2)\,\frac{\diff y}{y}.
\]
\end{thm}

\begin{proof}
We first check that for real roots $\alpha$ to $\Xi(\alpha)=0$, the indicated function $u_\alpha$
solves the indicated second order boundary value problem. By the product rule and the fundamental
theorem of calculus, we have that
\begin{multline}
\Dop^\times u_\alpha(t)=\frac{t}{\imag}u_\alpha'(t)
\\
=\frac{t}{\imag}\bigg(
-\imag\alpha t^{-\imag\alpha-1}\int_{0+}^{t}y^{\imag \alpha}\,
\varTheta_{00}(\imag y^2)\,\frac{\diff y}{y}+\frac{1}{t}\varTheta_{00}(\imag t^2)\bigg),
\end{multline}
so that, as a consequence,
\begin{equation}
\Dop^\times u_\alpha(t)+\alpha u_\alpha(t)=\frac{1}{\imag}\, \varTheta_{00}(\imag t^2).
\end{equation}
Next, since $\varTheta_{00}(\imag t^2)=\exp(-\varphi_{00}(t))$, we find by the chain rule that 
\begin{multline}
\Lop_{\phi_{00}}\Dop^\times u_\alpha(t)+\alpha \Lop_{\phi_{00}}u_\alpha(t)
=\frac{1}{\imag}\, \Lop_{\phi_{00}}(\varTheta_{00}(\imag t^2))
\\
=\frac{1}{\imag}\, \Lop_{\phi_{00}}(\e^{-\phi_{00}(t)})=\frac{1}{\imag}
\big(\Dop^\times\e^{-\phi_{00}(t)}+\e^{-\phi_{00}(t)}\Dop^\times\phi_{00}(t)\big)=0
\end{multline}
by the chain rule.
We should also verify the boundary values at $0^+$ and $+\infty$. To this end, we
first observe that since $\alpha\in\R$, we have
\begin{equation}
\lim_{t\to0^+}|u_\alpha(t)|=\lim_{t\to0^+}\bigg|\int_{0+}^{t}y^{\imag \alpha}\,
\varTheta_{00}(\imag y^2)\,\frac{\diff y}{y}\bigg|=0,
\end{equation}
and, similarly, 
\begin{multline}
\lim_{t\to+\infty}|u_\alpha(t)|=\lim_{t\to+\infty}\bigg|\int_{0+}^{t}y^{\imag \alpha}\,
\varTheta_{00}(\imag y^2)\,\frac{\diff y}{y}\bigg|
\\
=\bigg|\int_{0+}^{+\infty}y^{\imag \alpha}\,
\varTheta_{00}(\imag y^2)\,\frac{\diff y}{y}\bigg|=|\Xi(\alpha)|=0,
\end{multline}
by \eqref{eq:Xi-1}.

Next, we obtain the uniqueness of the solution. As the equation may be written in the
form
\begin{equation}
\Lop_{\phi_{00}}(\Dop^\times u+\alpha u)=0,
\label{eq:ODE-1}
\end{equation}
we put $v:=\Dop^\times u+\alpha u$, so that the equation asserts that 
$\Lop_{\phi_{00}}v=0$ on $\R_{>0}$. We should find the general solution to the latter
equation. The identity
\[
\Lop_{\phi_{00}}f=\e^{-\phi_{00}}\Dop^\times(\e^{\phi_{00}}f)
\]
shows that the general solution to $\Lop_{\phi_{00}}v=0$ is $v=C'\,\e^{-\phi_{00}}$
for some constant $C'$. But then \eqref{eq:ODE-1} implies that
\begin{equation}
\Dop^\times u+\alpha u=v=C'\,\e^{-\phi_{00}}
\label{eq:ODE-2}
\end{equation}
holds for some constant $C'$. Since generally
\[
\Dop^\times u(t)+\alpha u(t)=t^{-\imag\alpha}\Dop^\times(t^{\imag\alpha}u(t))
\]
it now follows from \eqref{eq:ODE-2} that
\[
\Dop^\times(t^{\imag\alpha}u(t))=C'\,t^{\imag\alpha}\e^{-\phi_{00}(t)}
=C'\,t^{\imag\alpha}\varTheta_{00}(\imag t^2)
\]
on the positive half-line $\R_{>0}$. Since $\alpha$ is real and $u(0^+)=0$ is 
assumed, we must have 
\[
t^{\imag\alpha}u(t)=\imag\,C'\int_{0^+}^{t}y^{\imag\alpha}\varTheta_{00}(\imag y^2)
\frac{\diff y}{y}
\]
so that with $C:=\imag\,C'$ we have
\begin{equation}
u(t)=C t^{-\imag\alpha}
\int_{0^+}^{t}y^{\imag\alpha}\varTheta_{00}(\imag y^2)
\frac{\diff y}{y}.
\label{eq:ODE-3}
\end{equation}
The other boundary condition, $u(+\infty)=0$, then amounts to having
\[
C\int_{0^+}^{+\infty}y^{\imag\alpha}\varTheta_{00}(\imag y^2)
\frac{\diff y}{y}=0,
\]
which by \eqref{eq:Xi-1} is the same as $C\Xi(\alpha)=0$. Consequently, 
either $C=0$ and hence $u(t)\equiv0$, or $\alpha$ is a root
of the equation $\Xi(\alpha)=0$, in which case $u$ must be of the given form 
$u=C u_\alpha$, as follows from the formula \eqref{eq:ODE-3}.
This completes the proof of the theorem.
\end{proof}

\begin{rem}
$(a)$ The fact that there is a unique eigenfunction at each eigenvalue
lends some weak support to the hypothesis that the real zeros of $\Xi$ are all simple. 

\noindent$(b)$ The commutator calculation 
\begin{multline*}
\Lop_{\phi_{00}}\Dop^\times f-\Dop^\times\Lop_{\phi_{00}}f=(\Dop^\times)^2f
+\Dop^\times\phi_{00}\Dop^\times f
\\
-(\Dop^\times)^2f-\Dop^\times(f\Dop^\times\phi_{00})
=-[(\Dop^\times)^2\phi_{00}]f
\end{multline*}
shows that the differential equation of Theorem \ref{thm:main-1} may be written in the
form
\begin{equation}
\Dop^\times \Lop_{\phi_{00}}u-[(\Dop^\times)^2\phi_{00}]u+\alpha\Lop_{\phi_{00}}u=0
\quad\text{on}\,\,\,\R_{>0}.
\label{eq:commutator}
\end{equation}
This may be useful to us as we would like to express the equation in terms of the 
function $\Lop_{\phi_{00}}u$ as much as possible. In \eqref{eq:commutator}, we express 
the difficulty of doing this in lowest order term form.  The second derivative function 
$(\Dop^\times)^2\phi_{00}$ which appears measures the convexity of $\phi_{00}$ is 
logarithmic coordinates.
\end{rem}

There is a variant of the theorem which applies to general complex $\alpha$. As the 
proof follows along the same lines, we omit the necessary details.

\begin{thm}
\label{thm:main-1.5}
For given $\alpha\in\C$, the solution  to the boundary value problem 
\begin{equation*}
\begin{cases}
\Lop_{\phi_{00}}\Dop^\times u+\alpha \Lop_{\phi_{00}}u=0\quad\text{on}\,\,\,\R_{>0},
\\
\lim_{t\to0^+}t^{-\im\alpha}u(t)=\lim_{t\to+\infty}t^{-\im\alpha}u(t)=0,
\end{cases}
\end{equation*}
is unique, i.e., $u=0$, unless $\alpha$ is such that $\Xi(\alpha)=0$. Moreover, if 
$\Xi(\alpha)=0$, then the solution is unique up to multiplicative constants, and of the form
$u=Cu_\alpha$, where $C$ is a constant and 
\[
u_\alpha(t)=t^{-\imag \alpha}\int_{0+}^{t}y^{\imag \alpha}\,
\varTheta_{00}(\imag y^2)\,\frac{\diff y}{y}.
\]
\end{thm}

\begin{rem}
For any $\alpha\in\C$ with $\Xi(\alpha)=0$, we know that $|\im\alpha|<\frac12$,
and it is easy to verify that 
\begin{equation}
\lim_{t\to0^+}\forall N\in\Z_{>0}:\quad t^{-N}u_\alpha(t)=\lim_{t\to+\infty}t^N u_\alpha(t)=0.
\label{eq:Ncond}
\end{equation}
This allows us to replace the boundary condition in 
Theorem \ref{thm:main-1.5} by the stronger condition \eqref{eq:Ncond}. The advantage 
would then be that the boundary condition no longer involves the parameter $\alpha$. 
\end{rem}

\section{In search of inner product structure}

\subsection{Self-adjointness}
The well-known Hilbert-P\'olya conjectures asks for a spectral interpretation of the 
zeros of $\Xi(\alpha)=0$. In brief, if we could find an unbounded self-adjoint operator 
$\Tope$ on a Hilbert space, its eigenvalues will automatically be real, and if those 
eigenvalues would coincide with the zeros $\alpha$ of $\Xi(\alpha)=0$, then the Riemann 
hypothesis would follow. Theorem \ref{thm:main-1} does supply an eigenvalue interpretation,
but as it is not in terms of a single unbounded operator it is unclear what would be a
proper substitute.

\subsection{In search of an appropriate sesquilinear form}
We let $\Hspace_{\mathrm{finite}}$ denote the space of all functions
\[
u(t)=\sum_{\alpha:\Xi(\alpha)=0}c(\alpha)\,u_\alpha(t)
\]
provided that the sums are actually finite, that is, $c(\alpha)=0$ holds except for finitely
many $\alpha$. Here, we include also possible roots off the real line in the strip 
$|\im\alpha|<\frac12$. In this instance, although we have studied them in any detail, 
the eigenfunctions $u_\alpha$ are given by the same formula as in Theorem 
\ref{thm:main-1}. It makes sense to consider a sesquilinear inner product of the form
\begin{equation}
\langle u,v\rangle_{1}:=\langle \Lop_{\phi_{00}}u, \Lop_{\phi_{00}}v\rangle_2,
\qquad u,v\in\Hspace_{\mathrm{finite}},
\label{eq:IP-1&2}
\end{equation}
where the inner product $\langle\cdot,\cdot\rangle_2$ is associated with a Hilbert space.
 This would mean that $\|\cdot\|_1$ defines a Hilbert seminorm, in the sense that the vector
 $\exp(-\phi_{00})$ has length $0$, since $\Lop_{\phi_{00}}(\exp(-\phi_{00}))=0$ 
 automatically.
 
\begin{defn}
The second order boundary value problem of Theorem \ref{thm:main-1} involving the pair
of differential operators $(\Lop_{\phi_{00}}\Dop^\times,\Lop_{\phi_{00}})$ is said to be
\emph{self-adjoint} with respect to the sesquilinear form $\langle\cdot,\cdot\rangle_1$ connected
via \eqref{eq:IP-1&2}  if
\[
\langle \Lop_{\phi_{00}}\Dop^\times u,\Lop_{\phi_{00}}v\rangle_2=
\langle \Lop_{\phi_{00}} u,\Lop_{\phi_{00}}\Dop^\times v\rangle_2
\]
holds for all $u,v\in \Hspace_{\mathrm{finite}}$.
 \end{defn}
  
\begin{thm}
If we can find a pair of related sesquilinear forms $\langle\cdot,\cdot\rangle_1$ and 
$\langle\cdot,\cdot\rangle_2$, such that the operator pair
$(\Lop_{\phi_{00}}\Dop^\times,\Lop_{\phi_{00}})$ is 
self-adjoint with respect to $\langle\cdot,\cdot\rangle_1$, we have the following.  

\noindent $(a)$ If $\alpha\notin\R$ and $\Xi(\alpha)=0$, then 
\[
\|u_\alpha\|_1^2=\langle u_\alpha,u_\alpha\rangle_1=
\langle \Lop_{\phi_{00}}u_\alpha, \Lop_{\phi_{00}}u_\alpha\rangle_2=0.
\]

\noindent $(b)$ If $\alpha,\beta\in\R$ with $\alpha\ne\beta$ and 
$\Xi(\alpha)=\Xi(\beta)=0$, then
\[
\langle u_\alpha,u_\beta\rangle_1=
\langle \Lop_{\phi_{00}}u_\alpha, \Lop_{\phi_{00}}u_\beta\rangle_2=0.
\]
\label{thm:main-2}
\end{thm}

\begin{proof}
We first consider part $(a)$. We have a point $\alpha\notin\R$ with $\Xi(\alpha)=0$. 
We apply the self-adjointness property to $u=v=u_\alpha$, which says that
\[
\langle \Lop_{\phi_{00}}\Dop^\times u_\alpha,\Lop_{\phi_{00}}u_\alpha\rangle_2=
\langle \Lop_{\phi_{00}} u_\alpha,\Lop_{\phi_{00}}\Dop^\times u_\alpha\rangle_2,
\]
and since $u_\alpha$ solves the differential equation 
\[
\Lop_{\phi_{00}}\Dop^\times u_\alpha+\alpha\Lop_{\phi_{00}}u_\alpha=0,
\]
it follows that 
\[
-\alpha\langle \Lop_{\phi_{00}}u_\alpha,\Lop_{\phi_{00}}u_\alpha\rangle_2=
-\bar\alpha\langle \Lop_{\phi_{00}} u_\alpha,\Lop_{\phi_{00}} u_\alpha\rangle_2.
\]
Since, by assumption, $\alpha\ne\bar\alpha$, this is only possible if 
\[
\langle u_\alpha,u_\alpha\rangle_1=
\langle \Lop_{\phi_{00}}u_\alpha,\Lop_{\phi_{00}}u_\alpha\rangle_2=0,
\]
as claimed.

We next turn to part $(b)$. We apply the self-adjointness property to the $u=u_\alpha$
and $v=u_\beta$, where $\alpha,\beta\in\R$ with $\alpha\ne\beta$ and 
$\Xi(\alpha)=\Xi(\beta)=0$, which asserts that
\[
\langle \Lop_{\phi_{00}}\Dop^\times u_\alpha,\Lop_{\phi_{00}}u_\beta\rangle_2=
\langle \Lop_{\phi_{00}} u_\alpha,\Lop_{\phi_{00}}\Dop^\times u_\beta\rangle_2.
\]
If we use that $u_\alpha$ and $u_\beta$ solves the respective differential equation,
it follows that
\[
-\alpha\langle \Lop_{\phi_{00}}u_\alpha,\Lop_{\phi_{00}}u_\beta\rangle_2=
-\beta\langle \Lop_{\phi_{00}} u_\alpha,\Lop_{\phi_{00}} u_\beta\rangle_2,
\]
and since $\alpha\ne\beta$ is assumed, we obtain
\[
\langle u_\alpha,u_\beta\rangle_1=
\langle \Lop_{\phi_{00}}u_\alpha,\Lop_{\phi_{00}}u_\beta\rangle_2=0.
\]
This completes the proof of the theorem.
\end{proof}

\begin{cor}
Under the assumptions of Theorem \ref{thm:main-2}, and provided that the inner
product $\langle\cdot,\cdot\rangle_2$ comes from a Hilbert space, it follows that
all roots of $\Xi(\alpha)=0$ are real.
\end{cor}

\begin{proof}
We apply part $(a)$ of Theorem \ref{thm:main-2}. If $\langle\cdot,\cdot\rangle_2$
comes from a Hilbert space, then by Theorem \ref{thm:main-2}$(a)$,  we know that
\[
\|\Lop_{\phi_{00}}u_\alpha\|_2^2=\langle\Lop_{\phi_{00}}u_\alpha,\Lop_{\phi_{00}}u_\alpha
\rangle_2=0
\] 
holds for any $\alpha\notin\R$ with $\Xi(\alpha)=0$. This entails that 
$\Lop_{\phi_{00}}u_\alpha=0$ and hence $u_\alpha=C\,\exp(-\phi_{00})$ must
hold for some constant $C$. This is impossible, which gives the desired contradiction. 
\end{proof}

\section{Remarks on the Jacobi identity and Appell transformation}

Having observed that the classical Appell transformation for the heat equation visually 
looks so much like the Jacobi identity \eqref{eq:heat-2}, and the connection does not 
appear to be understood in the literature, we decided to add a comment here.  

\subsection{Appell transformation}
The \emph{Appell transformation}  (see, e.g., \cite{Appell}, \cite{Widder}, \cite{Torre}), 
is a well-known transformation which preserves solutions to the heat equation.
If $u(z,\tau)$ solves the \eqref{eq:heat-1}, the Appell transformation 
sends it to the function $v(z,\tau)$ which solves the same heat equation, and it is given
by
\begin{equation}
v(z,\tau)= \heatkernel(z,\tau)\,u\bigg(\frac{\imag z}{\tau},\frac{1}{\tau}\bigg),
\end{equation}
where we recall the heat kernel $\heatkernel(z,\tau)$ given by \eqref{eq:heat-3}.


\subsection{Adding time reversion and a complex rotation in space}
We note that the change-of-variables $(z,\tau)\mapsto(-\imag z,-\tau)$ also preserves 
the heat equation \eqref{eq:heat-1}, so that
\begin{equation*}
\tilde v(z,\tau):=
v(-\imag z,-\tau)= \heatkernel(-\imag z,-\tau)\,u\bigg(\frac{z}{\tau},-\frac{1}{\tau}\bigg)
=\frac{1}{\imag}\heatkernel(z,\tau)\,u\bigg(\frac{z}{\tau},-\frac{1}{\tau}\bigg)
\end{equation*}
is yet another solution of the heat equation. Note that since the expression for the heat kernel
\eqref{eq:heat-3} involves a square root, the last equality involves a choice of that square root.
From this point of view, we may think of the Jacobi identity \eqref{eq:heat-2} as saying that 
when $u(z,\tau)=\vartheta_{00}(z,\tau)$, we get that 
$\tilde v(z,\tau)=-\imag \vartheta_{00}(z,\tau)$ is the associated modified Appell transformation 
of $u(z,\tau)$.



\end{document}